\documentclass[11pt]{article}
\usepackage{amsmath, amsthm, amssymb}
\usepackage{enumerate}
\usepackage[top=30truemm, bottom=30truemm, left=25truemm, right=25truemm]{geometry}
\usepackage[dvipdfm]{graphicx}
\usepackage{accents}
\usepackage{array}
\usepackage{bm}
\usepackage{caption}
\usepackage{cases}
\usepackage{comment}
\usepackage{float}
\usepackage{tikz-cd}
\usepackage{titlesec}
\usepackage[all]{xy}

\makeatletter
\def\sbar{\accentset{{\cc@style\underline{\mskip9mu}}}}
\def\mbar{\accentset{{\cc@style\underline{\mskip12mu}}}}
\def\lbar{\accentset{{\cc@style\underline{\mskip18mu}}}}

\makeatother
\newcommand{\keywords}[1]{\textbf{{Keywords: }} #1}
\newcommand{\MSC}[1]{\textbf{{2010 Mathematical Subject Classification: }} #1}
\titleformat*{\section}{\normalfont\large\bfseries}

\theoremstyle{definition}
\newtheorem{Def}{Definition}[section]
\newtheorem{Thm}[Def]{Theorem}
\newtheorem{Prop}[Def]{Proposition}
\newtheorem{Lem}[Def]{Lemma}
\newtheorem{Cor}[Def]{Corollary}
\newtheorem{Rmk}[Def]{Remark}
\newtheorem{Exp}[Def]{Example}
\title{\vspace{-5mm}On the maximality of hyperelliptic Howe curves of genus 3}
\author{Ryo Ohashi}

\begin{document}
\maketitle\vspace{-6mm}
\begin{abstract}
In this paper, we study a Howe curve $C$ in positive characteristic $p \geq 3$ which is of genus 3 and is hyperelliptic. We will show that if $C$ is superspecial, then its standard form is maximal or minimal over $\mathbb{F}_{p^2}\hspace{-0.3mm}$ without taking its $\mathbb{F}_{p^2}$-form.
\end{abstract}

\par
\MSC{14G05, 14G17, 14H45, 14H50}

\section{Introductrion}
Throughout this paper, a curve always means a projective variety in positive characteristic $p \geq 3$ of dimension one. It is well-known that any nonsingular genus-$g$ curve $C$ defined over $\mathbb{F}_q$ with $q = p^n$ satisfies the Hasse-Weil inequality
\[
	1+ q - 2g\sqrt{q} \leq \#C(\mathbb{F}_q) \leq 1 + q  + 2g\sqrt{q},
\]
where $C(\mathbb{F}_q)$ denotes the set of $\mathbb{F}_q$-rational points of $C$. Now, a nonsingular curve $C$ is called \textit{maximal} (resp.\ \textit{minimal}) when the number of $\mathbb{F}_q$-rational points of $C$ attains the upper (resp.\ lower) bound. Maximal curves have been investigated for their applications to coding theory. On the other hand, a genus-$g$ curve $C$ is called \textit{superspecial} when ${\rm Jac}(C)$ is isomorphic to the product of supersingular elliptic curves. This is equivalent to saying that the $a$-number of $C$ is equal to $g$. Here the $a$-number of $C$ is defined to be the dimension of ${\rm Hom}(\alpha_p,{\rm Jac}(C)[p])$, where $\alpha_p$ is the Frobenius kernel on the additive group $\mathbb{G}_a$. It is also known that all maximal or minimal curves over $\mathbb{F}_{p^2}$ are superspecial, while superspecial curves over $\mathbb{F}_{p^2}$ are not necessarily maximal nor minimal. \par

Ibukiyama \cite{Ibukiyama} showed that there exists a superspecial genus-3 curve $C$ over $\mathbb{F}_p$ such that $C$ is maximal over $\mathbb{F}_{p^2}$. However, the curve $C$ is not necessarily hyperelliptic. Kodama-Top-Washio \cite{Kodama} studied the maximality of some specific hyperelliptic curves of genus 3. In this paper, we study the maximalty of hyperelliptic Howe curves of genus 3. A Howe curve is the desingularization of the fiber product over $\mathbb{P}^1$ of two elliptic curves. See Section 2, for the precise definition of Howe curves. Such curves were studied in genus 4 by Howe \cite{Howe} to construct quickly curves having many rational points, and were called Howe curves in Kudo-Harashita-Senda \cite{Kudo}. Before this, Oort constructed a hyperelliptic superspecial curve of genus 3 by using a Howe curve in \cite{Oort}.\par
We see a hyperelliptic Howe curve $C$ of genus 3 is written as $y^2 = (x^4 - ax^2 + 1)(x^4 - bx^2 + 1)$ in Section 3, and we call it the standard form of $C$. Our main result is as below:
\begin{Thm}
Assume that $C: y^2 = (x^4 - ax^2 + 1)(x^4 - bx^2 + 1)$ is nonsingular and superspecial. Then $a$ and $b$ belong to $\mathbb{F}_{p^2}$. Moreover, the followings are true:
\begin{itemize}
\item If $p \equiv 3 \pmod{4}$, then $C$ is maximal over $\mathbb{F}_{p^2}$.\vspace{-1mm}
\item If $p \equiv 1 \pmod{4}$, then $C$ is minimal over $\mathbb{F}_{p^2}$.
\end{itemize}
\end{Thm}\vspace{-1mm}
\noindent See Corollary 4.4, for the maximality of the curve of the form $\varepsilon y^2 = (x^4 - ax^2 + 1)(x^4 - bx^2 + 1)$.

\newpage
The remainder of this paper is structured as follows. In Section 2, we review the definition of Howe curves. Proposition 2.4 gives the new condition that Howe curves of genus 3 are hyperelliptic. In Secion 3, we look into the structures of a hyperelliptic Howe curve $C$ of genus 3. In particular, we describe explicitly the elliptic curves appearing as quotients of $C$ by involutions. In Section 4, we prove Theorem 1.1. See Example 4.3, for the number of superspecial hyperelliptic Howe curves in small characteristic.\vspace{-1mm}

\subsection*{Acknowledgments}\vspace{-1mm}
This paper was written while the author is a Ph.D.\,student at Yokohama National University, and I would like to express my special thanks to my supervisor Prof.\ Shushi Harashita for his guidance. The author thank Toshiyuki Katsura and Tomoyoshi Ibukiyama for their helpful comments.

\section{Howe curve of genus 3}
In this section, we recall the definition of a Howe curve $C$ of genus 3 and find conditions that $C$ is hyperelliptic. Let $K$ be a perfect field of characteristic $p \geq 3$. Consider two elliptic curve
\begin{align*}
	E_1 &: y^2 = x(x-1)(x-\lambda_1),\\
	E_2 &: y^2 = x(x-\mu)(x-\mu\lambda_2)\\[-6mm]
\end{align*}
where $\lambda_1,\lambda_2$ and $\mu$ are not $0$ nor $1$. This condition is a condition for $E_1$ and $E_2$ to be nonsingular. Let $f_i: E_i \rightarrow \mathbb{P}^1$ be the usual double cover for $i \in \{1,2\}$. Here $f_1$ is ramified over $S_1= \{0,1,\lambda_1,\infty\}$ and $f_2$ is ramified over $S_2= \{0,\mu,\mu\lambda_2,\infty\}$. The fiber product $E_1 \hspace{-0.3mm}\times_{\mathbb{P}^1\!} E_2$ is irreducible if and only if $|S_1 \cap S_2| \leq 3$, and then the desingularization $C$ of the fiber product $E_1 \hspace{-0.3mm}\times_{\mathbb{P}^1\!} E_2$ is called a \textit{Howe curve}, see \cite[Section 4]{Takashima}. It is known \cite[Proposition 6.1]{Katsura} that the genus of a Howe curve $C$ is equal to $5 - |S_1 \cap S_2|$. More precisely, we obtain the followings:
\begin{itemize}
\item If $\mu = \lambda_1$ and $\mu\lambda_2 = 1$, then $C$ is not irreducible since $|S_1 \cap S_2| = 4$. 
\item Else if $\mu\lambda_2 = 1,\lambda_1$ or $\lambda_1\lambda_2$, then $C$ is a genus-2 curve since $|S_1 \cap S_2| = 3$.
\item Otherwise (i.e. $1,\lambda_1,\mu$ and $\mu\lambda_2$ are all different), then $C$ is a genus-3 curve since $|S_1 \cap S_2| = 2$.
\end{itemize}
We are interested in genus-3 curves, so we make the assumption that $1,\lambda_1,\mu$ and $\mu\lambda_2$ are all different from now on. Remark that nonhyperelliptic Howe curves of genus 3 are called Ciani curves (cf.\ \cite{Ciani}). Then, we have the following diagram\vspace{-3mm}
\begin{equation}
	\vcenter{
	\xymatrix{
	& C \ar[ld] \ar[d] \ar[rd] &\\
	E_1 \ar[rd]_{f_1} & E_3 \ar[d] & E_2 \ar[ld]^{f_2}\\
	& \mathbb{P}^1 &
	}}\vspace{-2mm}
\end{equation}
where $E_3$ is defined by
\[
	E_3: y^2 = (x-1)(x-\lambda_1)(x-\mu)(x-\mu\lambda_2).
\]
The diagram (2.1) induces an isogeny ${\rm Jac}(C) \rightarrow E_1 \hspace{-0.3mm}\times\hspace{-0.3mm} E_2 \hspace{-0.3mm}\times\hspace{-0.3mm} E_3$ by \cite[Section 3]{Kani} of degree $2^3$, thus we have ${\rm Jac}(C)[p] \cong (E_1 \hspace{-0.3mm}\times\hspace{-0.3mm} E_2 \hspace{-0.3mm}\times\hspace{-0.3mm} E_3)[p]$ since the degree of the isogeny is coprime to $p$. Hence $C$ is superspecial if and only if $E_i$ are all supersingular. Let us transform $E_3$ into the Legendre form: 
\begin{Lem}
The elliptic curve $E_3: y^2 = (x-1)(x-\lambda_1)(x-\mu)(x-\mu\lambda_2)$ is isomorphic to
\begin{equation}
	y^2 = x(x-1)(x-\lambda_3), \quad \lambda_3 = \frac{(\mu\lambda_2-1)(\mu-\lambda_1)}{(\mu\lambda_2-\lambda_1)(\mu-1)}.
\end{equation}
Note that $\lambda_3 \notin \{0,1,\infty\}$ is just the condition that $\lambda_1,\lambda_2$ and $\mu$ are not $0$ nor $1$.
\begin{proof}
The change of variables
\[
	X = \frac{x-1}{x-\lambda_1} \cdot \frac{\mu-\lambda_1}{\mu-1}, \quad Y = \frac{(\lambda_1-1)(\mu-\lambda_1)}{(\mu-1)^2} \cdot \frac{y}{(x-\lambda_1)^2}
\]
gives an equation
\[
	\frac{\mu\lambda_2-\lambda_1}{\mu-1} \cdot Y^2 = X(X-1)(X-\lambda_3).
\]
This is isomorphic to the elliptic curve in the Legendre form $y^2 = x(x-1)(x-\lambda_3)$ clearly.
\end{proof}
\end{Lem}

We have seen how the Legendre forms $y^2 = x(x-1)(x-\lambda_i)$ of $E_1,E_2$ and $E_3$ are obtained from the Howe curve associated to $\lambda_1, \lambda_2$ and $\mu$. Conversely we can find $\mu$ from $\lambda_1, \lambda_2$ and $\lambda_3$. Recall the definition of $\lambda_3$ in (2.2), then we obtain
\[
	(\mu\lambda_2-1)(\mu-\lambda_1) = \lambda_3(\mu\lambda_2-\lambda_1)(\mu-1)
\]
and this equation can be rearranged as
\begin{equation}
	\mu^2\lambda_2(1-\lambda_3) - \mu(\lambda_1\lambda_2-\lambda_2\lambda_3-\lambda_3\lambda_1+1) + \lambda_1(1-\lambda_3) = 0.
\end{equation}
Solving this quadratic equation for $\mu$, we get
\[
	\mu = \frac{(\lambda_1\lambda_2-\lambda_2\lambda_3-\lambda_3\lambda_1+1) \pm \hspace{-0.3mm}\sqrt{D}}{2\lambda_2(1-\lambda_3)}
\]
with $D = (\lambda_1\lambda_2-\lambda_2\lambda_3-\lambda_3\lambda_1+1)^2 - 4\lambda_1\lambda_2(1-\lambda_3)^2$.

\begin{Prop}
A Howe curve $C$ is hyperelliptic if and only if $\mu^2\lambda_2 = \lambda_1$.
\begin{proof}
See \cite[Section 5.11]{Oort}.
\end{proof}
\end{Prop}
Here is a paraphrase of the criterion in terms of the discriminant $D$. 
\begin{Prop}
A Howe curve $C$ is hyperelliptic if and only if $D=0$.
\begin{proof}
Suppose that $D=0$, then it follows from $(\lambda_1\lambda_2-\lambda_2\lambda_3-\lambda_3\lambda_1+1)^2 = 4\lambda_1\lambda_2(1-\lambda_3)^2$ that
\[
	\mu^2\lambda_2 = \frac{(\lambda_1\lambda_2-\lambda_2\lambda_3-\lambda_3\lambda_1+1)^2}{4{\lambda_2}(1-\lambda_3)^2} = \frac{4\lambda_1\lambda_2(1-\lambda_3)^2}{4{\lambda_2}(1-\lambda_3)^2} = \lambda_1.
\]
Therefore $C$ is hyperelliptic by Proposition 2.3. On the other hand, suppose that $C$ is hyperelliptic. By substituting $\mu^2\lambda_2 = \lambda_1$ for (2.3), we have
\[
	\mu(\lambda_1\lambda_2-\lambda_2\lambda_3-\lambda_3\lambda_1+1) = 2\lambda_1(1-\lambda_3).
\]
Square both sides and multiplying both sides by $\hspace{-0.3mm}\lambda_2$, then we have
\[
	\mu^2\lambda_2(\lambda_1\lambda_2-\lambda_2\lambda_3-\lambda_3\lambda_1+1)^2 = 4\lambda_1^2\lambda_2(1-\lambda_3)^2.
\]
Substituting $\mu^2\lambda_2 = \lambda_1$ again yields $ (\lambda_1\lambda_2-\lambda_2\lambda_3-\lambda_3\lambda_1+1)^2 = 4\lambda_1\lambda_2(1-\lambda_3)^2$.
This means $D=0$, thus the proposition is true.
\end{proof}
\end{Prop}

In the hyperelliptic case, only at most two values of $\lambda_3$ are possible for given $\lambda_1$ and $\lambda_2$ by using Proposition 2.3. The equation $D=0$ can be rearranged into
\[
	(\lambda_1-\lambda_2)^2\lambda_3^2 - 2\{\lambda_1(\lambda_2-1)^2+\lambda_2(\lambda_1-1)^2\}\lambda_3 + (\lambda_1\lambda_2-1)^2 = 0,
\]
and we can solve this equation as follows:
\begin{equation}
	\lambda_3 = \left\{
	\begin{array}{l}
		\displaystyle\frac{\{\hspace{-0.3mm}\sqrt{\lambda_1}(\lambda_2-1)\pm\hspace{-0.3mm}\sqrt{\lambda_2}(\lambda_1-1)\}^2}{(\lambda_1-\lambda_2)^2} \quad (\lambda_1 \neq \lambda_2),\\[3mm]
		\displaystyle\frac{(\lambda_1+1)^2}{4\lambda_1} \quad\hspace{35.7mm} (\lambda_1 = \lambda_2).
	\end{array}
	\right.
\end{equation}
This result will be used in Section 4.
\begin{Exp}
Consider the case that $\lambda_1 = \lambda_2$. Put $\lambda := \lambda_1 = \lambda_2$ and $\mu = -1$, that is,\vspace{-1mm}
\begin{align*}
	E_1 &: y^2 = x(x-1)(x-\lambda),\\
	E_2 &: y^2 = x(x+1)(x+\lambda),\\
	E_3 &: y^2 = (x^2-1)(x^2-\lambda^2).\\[-7mm]
\end{align*}
Then, the Howe curve $C$ is hyperelliptic by using Proposition 2.2. Moreover, the elliptic curve $E_3$ is isogenous to $Y^2 = X(X-1)(X-\lambda^2)$, and thus $C$ is superspecial if and only if $H_p(\lambda) = H_p(\lambda^2) = 0$ where we define\vspace{-1mm}
\[
	H_p(t) := \sum_{i = 0}^{(p-1)/2} \binom{(p-1)/2}{i}^{\!\!2}t^i.
\]
\end{Exp}

\section{Hyperelliptic Howe curve of genus 3}
\setcounter{equation}{0}
Let $K$ be a perfect field of characteristic $p \geq 3$. In this section, we consier the desingularization $C$ of a hyperelliptic curve of genus 3
\begin{equation}
	y^2 = x^8 + Mx^6 + Nx^4 + Mx^2 + 1
\end{equation}
at the infinity point with $M,N \in K$. We shall show that the curve $C$ is a hyperelliptic Howe curve of genus 3 if it is nonsingular and that any hyperelliptic Howe curve of genus 3 is realized as a curve of the form (3.1). Firstly, we study the nonsingularity of the curve $C$. Factorizing the right-hand side of (3.1), we get
\begin{equation}
	y^2 = (x^4 - ax^2 + 1)(x^4 - bx^2 + 1)
\end{equation}
where $a+b = -M$ and $ab = N-2$ holds. 
\begin{Lem}
The curve $C$ is nonsingular if and only if $a,b \neq \pm2$ and $a - b \neq 0$.
\begin{proof}
Recall that a hyperelliptic curve $y^2 = f(x)$ is nonsingular except at the infinity point if and only if $f(x)$ does not have any multiple root, so it suffices to show that $(x^4-ax^2+1)(x^4-bx^2+1)$ has a multiple root if and only if $a,b = \pm 2$ or $a-b = 0$.\par
Here, the polynomials $x^4-ax^2+1$ and $x^4-bx^2+1$ has a common root if and only if $a-b = 0$. Moreover, one can confirm that the polynomial $x^4-ax^2+1$ (resp.\ $x^4-bx^2+1$) has a multiple root if and only if $a = \pm2$ (resp.\ $b = \pm2$). The proof is completed.
\end{proof}
\end{Lem}
\begin{Rmk}
The curve $C$ is nonsingular if and only if
\[
	2M+N+2 \neq 0, \quad -2M+N+2 \neq 0, \quad M^2-4N+8 \neq 0.
\]
Indeed, by using $M = -(a+b)$ and $N = ab+2$, then we can compute $2M+N+2 = (a+2)(b+2)$, $-2M+N+2 = (a-2)(b-2)$ and $M^2-4N+8 = (a-b)^2$ directly.
\end{Rmk}
Next, let $\sigma_1,\sigma_2,\sigma_3$ be automorphisms be a nonsingular curve $C$ as below:
\begin{align*}
	\sigma_1: (x,y) \mapsto (1/x,y/x^4), \quad \sigma_2: (x,y) \mapsto (-1/x,y/x^4), \quad \sigma_3: (x,y) \mapsto (-x,y).
\end{align*}
Put $E_i := C/\langle\sigma_i\rangle$ for $i \in \{1,2,3\}$, then we obtain the equations of genus-1 curves
\begin{align*}
	E_1 &: v^2 = (u^2-a-2)(u^2-b-2), \quad\hspace{4.3mm} u = x+1/x,\ v = y/x^2,\\
	E_2 &: v^2 = (u^2-a+2)(u^2-b+2), \quad\hspace{4.3mm} u = x-1/x,\ v = y/x^2,\\
	E_3 &: v^2 = (u^2-au+1)(u^2-bu+1), \quad u = x^2,\ v = y\\[-6.5mm]
\end{align*}
and we can regard the curve
\[
	P: v^2 = (u-a)(u-b), \quad u = x^2+1/x^2,\ v = y/x^2
\]
as the quotient $C/\langle\sigma_1,\sigma_2\rangle$. The genus of $P$ is $0$, and thus the curve $P$ is isomorphic to the projective line $\mathbb{P}^1$. For the above discussion, we obtain the following diagram:
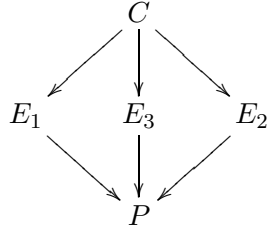
\vspace{-2mm}
\[
	\xymatrix{
	& C \ar[ld] \ar[d] \ar[rd] &\\
	E_1 \ar[rd] & E_3 \ar[d] & E_2 \ar[ld]\\
	& P &
	}\vspace{-2mm}
\]
This means that $C$ is a Howe curve of genus 3 as studied in Section 2. Conversely, any hyperelliptic Howe curve of genus 3 is written as (3.2) by \cite[Section 5]{Bouw}. We choose $\alpha_+,\alpha_-,\beta_+$ and $\beta_-$ such that
\begin{equation}
	(\alpha_+)^2 = a+2, \quad (\alpha_-)^2 = a-2, \quad (\beta_+)^2 = b+2, \quad (\beta_-)^2 = b-2.
\end{equation}
once and fix them throughout this paper.

\begin{Lem}
We can transform $E_i$ into the Legendre forms $y^2 = x(x-1)(x-\lambda_i)$ with
\begin{align}
	\lambda_1 &= \biggl(\frac{\alpha_+\hspace{-0.3mm}+\beta_+}{\alpha_+\hspace{-0.3mm}-\beta_+}\biggr)^{\!\!2},\nonumber\\
	\lambda_2 &= \biggl(\frac{\alpha_-\hspace{-0.3mm}-\beta_-}{\alpha_-\hspace{-0.3mm}+\beta_-}\biggr)^{\!\!2},\\
	\lambda_3 &= \biggl(\frac{\alpha_+\beta_-\hspace{-0.3mm}-\alpha_-\beta_+}{\alpha_+\beta_-\hspace{-0.3mm}+\alpha_-\beta_+}\biggr)^{\!\!2}.\nonumber
\end{align}
\begin{proof}
Factorizing the right hand side of the equation of $E_1$, we have
\[
	E_1: v^2 = (u+\alpha_+)(u-\alpha_+)(u+\beta_+)(u-\beta_+).
\]
\newpage \noindent The map
\[
	u \mapsto \frac{u-\alpha_+}{u+\alpha_+} \cdot \frac{\beta_+\hspace{-0.3mm}+\alpha_+}{\beta_+\hspace{-0.3mm}-\alpha_+}
\]
transforms $E_1$ into the Legendre form:
\[
	v^2 = u(u-1)\biggl(u-\frac{(\beta_+\hspace{-0.3mm}+\alpha_+)^2}{(\beta_+\hspace{-0.3mm}-\alpha_+)^2}\biggr) = u(u-1)(u-\lambda_1).
\]
Therefore, we obtain the equation of $\lambda_1$ in (3.4). The case of $E_2$ can be shown in the same way.\par
Secondly, we factorize the right-hand side of the equation of $E_3$:
\[
	E_3: v^2 = (u-\gamma_1)(u-\gamma_2)(u-\gamma_3)(u-\gamma_4)
\]
where
\[
	\gamma_1 = \frac{a+\alpha_+\alpha_-}{2}, \quad \gamma_2 = \frac{a-\alpha_+\alpha_-}{2}, \quad \gamma_3 = \frac{b + \beta_+\beta_-}{2}, \quad \gamma_4 = \frac{b - \beta_+\beta_-}{2}.
\]
The map
\[
	u \mapsto \frac{u - \gamma_1}{u-\gamma_2} \cdot \frac{\gamma_3-\gamma_2}{\gamma_3-\gamma_1}
\]
transforms $E_3$ into the Legendre form:
\begin{equation}
	v^2 = u(u-1)\biggl(u-\frac{(\gamma_4-\gamma_2)(\gamma_3-\gamma_1)}{(\gamma_4-\gamma_1)(\gamma_3-\gamma_2)}\biggr) = u(u-1)\biggl(u-\frac{\gamma_1\gamma_4+\gamma_2\gamma_3-\gamma_1\gamma_2-\gamma_3\gamma_4}{\gamma_1\gamma_3+\gamma_2\gamma_4-\gamma_1\gamma_2-\gamma_3\gamma_4}\biggr)
\end{equation}
Here, we obtain
\begin{align*}
	ab &= (\gamma_1+\gamma_2)(\gamma_3+\gamma_4) = \gamma_1\gamma_3 + \gamma_1\gamma_4 + \gamma_2\gamma_3 + \gamma_2\gamma_4,\\
	\alpha_+\alpha_-\beta_+\beta_- &= (\gamma_1-\gamma_2)(\gamma_3-\gamma_4) = \gamma_1\gamma_3-\gamma_1\gamma_4-\gamma_2\gamma_3 + \gamma_3\gamma_4.
\end{align*}
By adding or substracting these equalities, we have
\begin{equation}
	\gamma_1\gamma_3 + \gamma_2\gamma_4 = \frac{ab+\alpha_+\alpha_-\beta_+\beta_-}{2}, \quad \gamma_1\gamma_4 + \gamma_2\gamma_3 = \frac{ab-\alpha_+\alpha_-\beta_+\beta_-}{2}.
\end{equation}
Substituting (3.6) and $\gamma_1\gamma_2 = \gamma_3\gamma_4 = 1$ for (3.5), we obtain the equation of $\lambda_3$ in (3.4).
\end{proof}
\end{Lem}

As square roots of $\lambda_1$ and $\lambda_2$, we choose
\begin{align}
\begin{aligned}
	\sqrt{\lambda_1} := \frac{\alpha_+\hspace{-0.3mm}+\beta_+}{\alpha_+\hspace{-0.3mm}-\beta_+},\\ \sqrt{\lambda_2} := \frac{\alpha_-\hspace{-0.3mm}-\beta_-}{\alpha_-\hspace{-0.3mm}+\beta_-}
\end{aligned}
\end{align}
and set $\hspace{-0.3mm}\sqrt{\lambda_1\lambda_2} := \sqrt{\lambda_1}\sqrt{\lambda_2}$ and $\hspace{-0.3mm}\sqrt{\lambda_1/\lambda_2} := \sqrt{\lambda_1}/\hspace{-0.3mm}\sqrt{\lambda_2}$.
\begin{Rmk}
Remark that $\hspace{-0.3mm}\sqrt{\lambda_1}-\hspace{-0.3mm}\sqrt{\lambda_2} \neq 0$ and $\hspace{-0.3mm}\sqrt{\lambda_1\lambda_2}-1 \neq 0$ by assumption of $a-b \neq 0$. Indeed, suppose that $\hspace{-0.3mm}\sqrt{\lambda_1}-\hspace{-0.3mm}\sqrt{\lambda_2} = 0$, then we obtain
\[
	(\alpha_+\hspace{-0.3mm}+\beta_+)(\alpha_-\hspace{-0.3mm}+\beta_-) = (\alpha_-\hspace{-0.3mm}+\beta_+)(\alpha_-\hspace{-0.3mm}-\beta_-)
\]
by (3.7). This equality can be simplified as $\hspace{-0.3mm}\alpha_+\beta_- = -\alpha_-\beta_+$. Squaring both sides, we have $a = b$. Hence, we obtain $\hspace{-0.3mm}\sqrt{\lambda_1}-\hspace{-0.3mm}\sqrt{\lambda_2} \neq 0$. We can prove $\hspace{-0.3mm}\sqrt{\lambda_1\lambda_2}-1 \neq 0$ in a similar way.
\end{Rmk}
\begin{Prop}
The reverse transformation of that in Lemma 3.3 is given by
\begin{align*}
	a &= \frac{2\bigl(\lambda_1\sqrt{\lambda_2}+\hspace{-0.3mm}\sqrt{\lambda_1}\lambda_2+4\sqrt{\lambda_1\lambda_2}+\hspace{-0.3mm}\sqrt{\lambda_1}+\hspace{-0.3mm}\sqrt{\lambda_2}\bigr)}{\bigl(\hspace{-0.3mm}\sqrt{\lambda_1}-\hspace{-0.3mm}\sqrt{\lambda_2}\bigr)\bigl(\hspace{-0.3mm}\sqrt{\lambda_1\lambda_2} - 1\bigr)},\\
	b &= \frac{2\bigl(\lambda_1\sqrt{\lambda_2}+\hspace{-0.3mm}\sqrt{\lambda_1}\lambda_2-4\sqrt{\lambda_1\lambda_2}+\hspace{-0.3mm}\sqrt{\lambda_1}+\hspace{-0.3mm}\sqrt{\lambda_2}\bigr)}{\bigl(\hspace{-0.3mm}\sqrt{\lambda_1}-\hspace{-0.3mm}\sqrt{\lambda_2}\bigr)\bigl(\hspace{-0.3mm}\sqrt{\lambda_1\lambda_2} - 1\bigr)}.
\end{align*}
\begin{proof}
By (3.7), we have
\begin{align*}
\begin{aligned}
	\sqrt{\lambda_1\lambda_2} = \frac{(\alpha_+\hspace{-0.3mm}+\beta_+)(\alpha_-\hspace{-0.3mm}-\beta_-)}{(\alpha_-\hspace{-0.3mm}-\beta_+)(\alpha_-\hspace{-0.3mm}+\beta_-)},\\\sqrt{\lambda_1/\lambda_2} = \frac{(\alpha_+\hspace{-0.3mm}+\beta_+)(\alpha_-\hspace{-0.3mm}+\beta_-)}{(\alpha_-\hspace{-0.3mm}-\beta_+)(\alpha_-\hspace{-0.3mm}-\beta_-)}.
\end{aligned}
\end{align*}
A tedious computation shows
\begin{align*}
\begin{aligned}
	\frac{\sqrt{\lambda_1\lambda_2}+1}{\sqrt{\lambda_1\lambda_2}-1} = \frac{\alpha_+\alpha_-\hspace{-0.3mm}-\beta_+\beta_-}{\alpha_-\beta_+\hspace{-0.3mm}-\alpha_+\beta_-},\\
	\frac{\sqrt{\lambda_1/\lambda_2}+1}{\sqrt{\lambda_1/\lambda_2}-1} = \frac{\alpha_+\alpha_-\hspace{-0.3mm}+\beta_+\beta_-}{\alpha_-\beta_+\hspace{-0.3mm}+\alpha_+\beta_-}.
\end{aligned}
\end{align*}
Hence, we have
\[
	\frac{\sqrt{\lambda_1\lambda_2}+1}{\sqrt{\lambda_1\lambda_2}-1} \cdot \frac{\sqrt{\lambda_1/\lambda_2}+1}{\sqrt{\lambda_1/\lambda_2}-1} = \frac{(a^2-4)-(b^2-4)}{(a-2)(b+2)-(a+2)(b-2)} = \frac{(a+b)(a-b)}{4(a-b)}.
\]
By assumption of $a-b \neq 0$, thus we obtain
\begin{equation}
	a+b = 4\cdot\frac{\bigl(\hspace{-0.3mm}\sqrt{\lambda_1}+\hspace{-0.3mm}\sqrt{\lambda_2}\bigr)\bigl(\hspace{-0.3mm}\sqrt{\lambda_1\lambda_2}+1\bigr)}{\bigl(\hspace{-0.3mm}\sqrt{\lambda_1}-\hspace{-0.3mm}\sqrt{\lambda_2}\bigr)\bigl(\hspace{-0.3mm}\sqrt{\lambda_1\lambda_2} - 1\bigr)}
\end{equation}

On the other hand, we can compute
\begin{align}
\begin{aligned}
	\sqrt{\lambda_1}-\hspace{-0.3mm}\sqrt{\lambda_2} = \frac{2(\alpha_-\beta_+\hspace{-0.3mm}+\alpha_+\beta_-)}{(\alpha_+\hspace{-0.3mm}-\beta_+)\bigl(\alpha_-\hspace{-0.3mm}+\beta_-)},\\
	\sqrt{\lambda_1\lambda_2} - 1 = \frac{2(\alpha_-\beta_-\hspace{-0.3mm}+\alpha_+\beta_-)}{(\alpha_+\hspace{-0.3mm}-\beta_+)\bigl(\alpha_-\hspace{-0.3mm}+\beta_-)}
\end{aligned}
\end{align}
and one can check that
\begin{equation}
	a-b = 16 \cdot \frac{\sqrt{\lambda_1\lambda_2}}{\bigl(\hspace{-0.3mm}\sqrt{\lambda_1}-\hspace{-0.3mm}\sqrt{\lambda_2}\bigr)\bigl(\hspace{-0.3mm}\sqrt{\lambda_1\lambda_2} - 1\bigr)}.
\end{equation}
The proposition follows from (3.8) and (3.10).
\end{proof}
\end{Prop}

\begin{Rmk}
By using Lemma 3.3 and (3.9), we obtain
\[
	\lambda_3 = \biggl(\frac{\sqrt{\lambda_1\lambda_2}-1}{\sqrt{\lambda_1}-\hspace{-0.3mm}\sqrt{\lambda_2}}\biggr)^{\!\!2}.
\]
One can see that this equation coincides with (2.4).
\end{Rmk}
\section{Proof of the main theorem}
\setcounter{equation}{0}
In this section, we will show the theorem stated in Introduction.
\setcounter{section}{1}
\begin{Thm}
Assume that $C: y^2 = (x^4 - ax^2 + 1)(x^4 - bx^2 + 1)$ is nonsingular and superspecial. Then $a$ and $b$ belong to $\mathbb{F}_{p^2}$. Moreover, the followings are true:
\begin{itemize}
\item If $p \equiv 3 \pmod{4}$, then $C$ is maximal over $\mathbb{F}_{p^2}$.\vspace{-1mm}
\item If $p \equiv 1 \pmod{4}$, then $C$ is minimal over $\mathbb{F}_{p^2}$.
\end{itemize}
In particular, the curve $C$ is maximal or minimal over $\mathbb{F}_{p^2}$.
\end{Thm}

\setcounter{section}{4}
\setcounter{Def}{0}
A key to the proof is the following result by Auer and Top \cite[Proposition 2.2]{Top}.
\begin{Prop}
Let $E: y^2 = x(x-1)(x-\lambda)$ be a supersingular elliptic curve, then $\lambda$ is a 4th power in $(\mathbb{F}_{p^2}\hspace{-0.3mm})^{\hspace{-0.3mm}\times}$. Moreover, the followings are true:
\begin{itemize}
\item If $p \equiv 3 \pmod{4}$, then a elliptic curve $E$ is maximal over $\mathbb{F}_{p^2}$.\vspace{-1mm}
\item If $p \equiv 1 \pmod{4}$, then a elliptic curve $E$ is minimal over $\mathbb{F}_{p^2}$.
\end{itemize}
In particular, the elliptic curve $E$ is maximal or minimal over $\mathbb{F}_{p^2}$.
\end{Prop}

Recall the discussions in Section 3. Let $E_1,E_2$ and $E_3$ be the following three elliptic curves:
\begin{align*}
	E_1 &: v^2 = (u^2-a-2)(u^2-b-2),\\
	E_2 &: v^2 = (u^2-a+2)(u^2-b+2),\\
	E_3 &: v^2 = (u^2-au+1)(u^2-bu+1).
\end{align*}
Then, there exist morphisms $C \rightarrow E_i$ of degree 2. The morphisms are defined over $\mathbb{F}_{p^2}$, since so are the involutions defining these quotient.

\begin{proof}[Proof of the first half of Theorem 1.1]
The elliptic curves $E_i$ are supersingular by the assumption since the quotient of supersingular curve is supersingular. By Lemma 3.3, each elliptic curve $E_i$ is isomorphic to $y^2 = x(x-1)(x-\lambda_i)$. By Proposition 4.1, we obtain $\lambda_i$ is a 4th power in $(\mathbb{F}_{p^2}\hspace{-0.3mm})^{\hspace{-0.3mm}\times}$ and thus $\sqrt{\lambda_i} \in \mathbb{F}_{p^2}$. Hence, it follows from Proposition 3.4 that $a,b \in \mathbb{F}_{p^2}$.
\end{proof}

We have shown that $C$ is defined over $\mathbb{F}_{p^2}$. Next, let us discuss whether $C$ is maximal or minimal over $\mathbb{F}_{p^2}$. Here, we prepare the following lemma:
\begin{Lem}
If $C$ is superspecial, then $\alpha_+,\alpha_-,\beta_+$ and $\beta_-$ in (3.3) are elements of $\mathbb{F}_{p^2}$.
\begin{proof}
The elliptic curve $E_i$ is supersingular by the assumption, thus $\lambda_1,\lambda_2$ and $\lambda_3$ in (3.4) are 4th powers in $(\mathbb{F}_{p^2}\hspace{-0.3mm})^\times$ by Proposition 4.1. Here, one can compute that
\begin{align*}
	a-2 &= \frac{4\sqrt{\lambda_1}(\hspace{-0.3mm}\sqrt{\lambda_2}+1)^2}{\bigl(\hspace{-0.3mm}\sqrt{\lambda_1}-\hspace{-0.3mm}\sqrt{\lambda_2}\bigr)\bigl(\hspace{-0.3mm}\sqrt{\lambda_1\lambda_2} - 1\bigr)}, \quad a+2 = \frac{4\sqrt{\lambda_2}(\hspace{-0.3mm}\sqrt{\lambda_1}+1)^2}{\bigl(\hspace{-0.3mm}\sqrt{\lambda_1}-\hspace{-0.3mm}\sqrt{\lambda_2}\bigr)\bigl(\hspace{-0.3mm}\sqrt{\lambda_1\lambda_2} - 1\bigr)},\\
	b-2 &= \frac{4\sqrt{\lambda_1}(\hspace{-0.3mm}\sqrt{\lambda_2}-1)^2}{\bigl(\hspace{-0.3mm}\sqrt{\lambda_1}-\hspace{-0.3mm}\sqrt{\lambda_2}\bigr)\bigl(\hspace{-0.3mm}\sqrt{\lambda_1\lambda_2} - 1\bigr)}, \hspace{0.5mm}\quad b+2 = \frac{4\sqrt{\lambda_2}(\hspace{-0.3mm}\sqrt{\lambda_1}-1)^2}{\bigl(\hspace{-0.3mm}\sqrt{\lambda_1}-\hspace{-0.3mm}\sqrt{\lambda_2}\bigr)\bigl(\hspace{-0.3mm}\sqrt{\lambda_1\lambda_2} - 1\bigr)}
\end{align*}
by using Proposition 3.4, and so it suffices to show that $\bigl(\hspace{-0.3mm}\sqrt{\lambda_1}-\hspace{-0.3mm}\sqrt{\lambda_2}\bigr)\bigl(\hspace{-0.3mm}\sqrt{\lambda_1\lambda_2} - 1\bigr)$ is a square in $\mathbb{F}_{p^2}$.

\newpage \noindent It follows from Remark 3.6 that 
\[
	\sqrt{\lambda_3} = \frac{\sqrt{\lambda_1\lambda_2}-1}{\sqrt{\lambda_1}-\hspace{-0.3mm}\sqrt{\lambda_2}}
\]
is a square in $\mathbb{F}_{p^2}$, and thus whether $\sqrt{\lambda_1}-\hspace{-0.3mm}\sqrt{\lambda_2}$ is a square in $\mathbb{F}_{p^2\hspace{-0.3mm}}$ is in accord with whether $\sqrt{\lambda_1\lambda_2}-1$ is a square in $\mathbb{F}_{p^2}$. This means that a product $\bigl(\hspace{-0.3mm}\sqrt{\lambda_1}-\hspace{-0.3mm}\sqrt{\lambda_2}\bigr)\bigl(\hspace{-0.3mm}\sqrt{\lambda_1\lambda_2} - 1\bigr)$ is a square in $\mathbb{F}_{p^2}$.
\end{proof}
\end{Lem}
\begin{proof}[Proof of the second half of Theorem 1.1]
In the proof of Lemma 3.3, we defined isomorphisms from the elliptic curve $E_i$ to its Legendre form $y^2 = x(x-1)(x-\lambda_i)$ by using $\alpha_+,\alpha_-,\beta_+$ and $\beta_-$. Hence, these isomorphisms are defined over $\mathbb{F}_{p^2\hspace{-0.3mm}}$ by Lemma 4.2. Therefore $C$ is maximal (resp.\ minimal) if and only if all $y^2 = x(x-1)(x-\lambda_i)$ are maximal (resp.\ minimal). Suppose that $C$ is superspecial, and the elliptic curves $E_1,E_2$ and $E_3$ are all supersingular.
\begin{itemize}
\item If $p \equiv 3 \pmod{4}$, then the elliptic curve $y^2 = x(x-1)(x-\lambda_i)$ is maximal by Proposition 4.1. Hence $C$ is maximal over $\mathbb{F}_{p^2}$.\vspace{-1mm}
\item If $p \equiv 1 \pmod{4}$, then the elliptic curve $y^2 = x(x-1)(x-\lambda_i)$ is minimal by Proposition 4.1. Hence $C$ is minimal over $\mathbb{F}_{p^2}$.
\end{itemize}
Therefore, the proof of Theorem 1.1 is completed.
\end{proof}

Finally, we give examples of superspecial hyperelliptic Howe curves.
\begin{Exp}
The automorphism group $G$ of a hyperelliptic Howe curve of genus 3 in characteristic $p > 7$ is either ${\rm C}_2 \hspace{-0.3mm}\times\hspace{-0.3mm} {\rm C}_2 \hspace{-0.3mm}\times\hspace{-0.3mm} {\rm C}_2,\,{\rm C}_2 \hspace{-0.3mm}\times\hspace{-0.3mm} {\rm D}_8, {\rm V}_8$ or ${\rm C}_2 \hspace{-0.3mm}\times\hspace{-0.3mm} {\rm S}_4$, see \cite[Section 3.1]{Lercier}:
\begin{enumerate}
\item[(i)] If $G \cong {\rm C}_2 \hspace{-0.3mm}\times\hspace{-0.3mm} {\rm C}_2 \hspace{-0.3mm}\times\hspace{-0.3mm} {\rm C}_2$, then $C$ is isomorphic to $y^2 = (x^4-ax^2+1)(x^4-bx^2+1)$ with $a+b \neq 0$;
\item[(ii)] If $G \cong {\rm C}_2 \hspace{-0.3mm}\times\hspace{-0.3mm} {\rm D}_8 $, then $C$ is isomorphic to $y^2 = x^8-ax^4+1$ with $a \neq 0, \pm 14$;
\item[(iii)] If $G \cong {\rm V}_8$, then $C$ is isomorphic to $y^2 = x^8 - 1$;
\item[(iv)] If $G \cong {\rm C}_2 \hspace{-0.3mm}\times\hspace{-0.3mm} {\rm S}_4$, then $C$ is isomorphic to $y^2 = x^8 \pm 14x^4 +1$.
\end{enumerate}
If $C$ is superspecial, then $C$ is maximal or minimal in every case. Indeed, apply Theorem 1.1, after we factorize the right hand side of the defining equation as the form of (i) over an algebraic closed field, where in the case (iii) we use $y^2=x^8+1$ instead of $y^2=x^8-1$. Remark that $y^2 = x^8-1$ is superspecial if and only if $p \equiv 7 \pmod{8}$, in case $y^2 = x^8 - 1$ and $y^2 = x^8 + 1$ are isomorphic.\par

Here, for $7 < p < 100$, the number of hyperelliptic superspecial Howe curves $C$ of genus 3 with automorphism group is as follows if the number is positive (cf. \cite[Theorem 3.15]{Brock}).
\begin{table}[H]
\centering
	\begin{tabular}{c||wc{9.6mm}|wc{9.6mm}|wc{9.6mm}|wc{9.6mm}|wc{9.6mm}|wc{9.6mm}|wc{9.6mm}|wc{9.6mm}|wc{9.6mm}|wc{9.6mm}}
	\hline
	${\rm Aut}(C)$ & $p=17$ & $p=23$ & $p=31$ & $p=41$ & $p=47$ & $p=71$ & $p=73$ & $p=79$ & $p=89$ & $p=97$\\\hline
	${\rm C}_2 \hspace{-0.3mm}\times\hspace{-0.3mm} {\rm C}_2 \hspace{-0.3mm}\times\hspace{-0.3mm} {\rm C}_2$ & 0 & 2 & 3 & 0 & 4 & 10 & 2 & 9 & 0 & 4\\
	${\rm C}_2 \hspace{-0.3mm}\times\hspace{-0.3mm} {\rm D}_8$ & 1 & 0 & 1 & 1 & 2 & 0 & 3 & 1 & 1 & 1\\
	${\rm V}_8$ & 0 & 1 & 1 & 0 & 1 & 1 & 0 & 1 & 0 & 0\\
	${\rm C}_2 \hspace{-0.3mm}\times\hspace{-0.3mm} {\rm S}_4$ & 0 & 0 & 0 & 0 & 1 & 0 & 0 & 0 & 0 & 0\\
	\hline
	\end{tabular}
\end{table}\vspace{-1mm}
We use Magma calculator to obtain the above table. We compute the Hasse-Witt matrix of the curve $C: y^2 = (x^4-ax^2+1)(x^4-bx^2+1)$, and find $a,b \in \mathbb{F}_{p^2}$ such that $C$ is superspecial. Then, we decide whether these curves are isomorphic by using Shioda invariants (cf.\ \cite{Shioda}).
\end{Exp}

\newpage
Lastly, let us discuss the maximality of twists of the curve $C$. Note that in general, a hyperelliptic curve $C: y^2 = f(x)$ is maximal (resp. minimal) over $\mathbb{F}_{p^{2e}}\!$ if and only if $C': \varepsilon y^2 = f(x)$ is minimal (resp. maximal) over $\mathbb{F}_{p^{2e}\!}$ for non-square $\varepsilon \in (\mathbb{F}_{p^{2e}}\hspace{-0.5mm})^{\hspace{-0.3mm}\times}$.
\begin{Cor}
Assume that a hyperelliptic curve $C: y^2 = (x^4-ax^2+1)(x^4-bx^2+1)$ is superspecial. Let $\varepsilon \in (\mathbb{F}_{p^{2e}}\hspace{-0.5mm})^{\hspace{-0.3mm}\times\!}$ and $C': \varepsilon y^2 = (x^4-ax^2+1)(x^4-bx^2+1)$.\vspace{-1mm}
\begin{enumerate}
\item When $e$ is odd, then the followings are true:\vspace{-1mm}
\begin{itemize}
\item If $p \equiv 3 \pmod{4}$, then $C'$ is maximal over $\mathbb{F}_{p^{2e}}\!$ if and only if $\varepsilon$ is a square in $\mathbb{F}_{p^{2e}}$.\vspace{-1mm}
\item If $p \equiv 1 \pmod{4}$, then $C'$ is maximal over $\mathbb{F}_{p^{2e}}\!$ if and only if $\varepsilon$ is not a square in $\mathbb{F}_{p^{2e}}$.\vspace{-1mm}
\end{itemize}
\item When $e$ is even, then $C'$ is maximal over $\mathbb{F}_{p^{2e}}\!$ if and only if $\varepsilon$ is not a square in $\mathbb{F}_{p^{2e}}$.
\end{enumerate}
\begin{proof}
Suppose that a curve $\varGamma$ over $\mathbb{F}_{p^2}$ is maximal (resp. minimal), then $\varGamma$ over $\mathbb{F}_{p^{2e}}\!$ is also maximal (resp. minimal) when $e$ is odd, and is minimal when $e$ is even. This follows from the Weil conjecture (cf. \cite[Appendix\hspace{1mm}C, Exercise 5.7]{Hartshorne}) and the fact that $\varGamma$ over $\mathbb{F}_{p^2}$ is maximal (resp. minimal) if and only if all the eigenvalues of the Frobenius on the first \'{e}tale cohomology group are $-p$ (resp. $p$).
\end{proof}
\end{Cor}


\begin{thebibliography}{99}
\bibitem{Top} R.\,Auer and J.\,Top: \textit{Legendre elliptic curves over finite fields}, Journal of Number Theory\ {\bf 95}, 303--312, 2002.
\bibitem{Bouw} I.\,Bouw,\ N.\,Coppola,\ P.\,Kilicer,\ S.\,Kunzweiler,\ E.\,L.\,Garc\'{i}a and A.\,Somoza: \textit{Reduction type of genus-3 curves in a special stratum of their moduli space}, arXiv:\,2003.07633, 2020.
\bibitem{Brock} B.\,W.\,Brock: \textit{Superspecial curves of genera two and three}, Thesis (Ph.D.)-Princeton University, 1993. 
\bibitem{Ciani} E.\,Ciani: \textit{I varii tipi possibili di quartiche piane pi\`{u} volte omologico-armoniche}, Rendiconti del Circolo Matematico di Palermo\ {\bf 13}, 347--373, 1899.
\bibitem{Hartshorne} R. Hartshorne: \textit{Algebraic Geometry}, GTM\,{\bf 52}, Springer--Verlag, 1977.
\bibitem{Howe} E.\,W.\,Howe: \textit{Quickly constructing curves of genus 4 with many points}, Contemporary Mathematics\ {\bf 663}, 149--173, 2016.
\bibitem{Ibukiyama} T.\,Ibukiyama: \textit{On rational points of curves of genus 3 over finite fields}, T\^{o}hoku Mathematical Journal\ {\bf 45}, 311--329, 1993.
\bibitem{Kani} E.\,Kani and M.\,Rosen: \textit{Idempotent relations and factors of Jacobians}, Mathematische Annalen\ {\bf 284}, 307--327, 1989.
\bibitem{Katsura} T.\,Katsura: \textit{Decomposed Richelot isogenies of Jacobian varieties of curves of genus 3}, Journal of Algebra\ {\bf 588}, 129--147, 2021.
\bibitem{Takashima} T.\,Katsura and K.\,Takashima:\ \textit{Decomposed Richelot isogenies of Jacobian varieties of hyperelliptic curves and generalized Howe curves}, arXiv:\,2108.06936, 2021.
\bibitem{Kodama} T.\,Kodama,\ J.\,Top and T.\,Washio: \textit{Maximal hyperelliptic curves of genus three}, Finite Fields and Their Applications\ {\bf 15}, 392--403, 2009.
\bibitem{Kudo} M.\,Kudo,\ S.\,Harashita and H.\,Senda: \textit{The existence of supersingular curves of genus 4 in arbitrary characteristic}, Research in Number Theory\ {\bf 6}, Article number:\ 44, 2020.
\bibitem{Lercier} R.\,Lercier and C.\,Ritzenthaler: \textit{Hyperelliptic curves and their invariants:\ geometric,\ arithmetic and algorithmic aspects}, Journal of Algebra\ {\bf 372}, 595--636, 2012.
\bibitem{Oort} F.\,Oort: \textit{Hyperelliptic supersingular curves}, Progress in Mathematics\ {\bf 89}, 247--284, 1991.
\bibitem{Shioda} T.\,Shioda: \textit{On the graded ring of invariants of binary octavics}, 
American Journal of Mathematics {\bf 89}, 1022--1046, 1967.
\end{thebibliography}
\end{document}